\documentclass[11pt]{amsart} 
\usepackage{latexsym} 
\usepackage{times}
\usepackage{amsmath}
\usepackage{amssymb} 
\usepackage{amsfonts}
\usepackage{graphicx}

\newtheorem{theorem}{Theorem}[section]

\newtheorem{lemma}[theorem]{Lemma}
\newtheorem{corollary}[theorem]{Corollary}

\newtheorem{question}[theorem]{Question}

\newcommand{\R}{{\mathbb R}}

\newcommand{\eps}{\varepsilon}

\renewcommand{\le}         {\leqslant}

\renewcommand{\ge}         {\geqslant}
\renewcommand{\geq}         {\geqslant}

\newcommand{\LLP}{{L_+}}
\newcommand{\LLM}{{L_-}}
\newcommand{\EEE}{{{\Lambda}}}
\newcommand{\EEEP}{{{\EEE}_+}}
\newcommand{\EEEM}{{{\EEE}_-}}

\usepackage{mathrsfs}
\renewcommand{\mathcal}{\mathscr}

\begin{document}

\title[$1$D symmetry for semilinear PDEs]{$1$D symmetry for
semilinear PDEs\\
from the limit interface of the solution}

\author{Alberto Farina}
\address{ {LAMFA -- CNRS UMR 6140}
\\ {Universit\'e de Picardie Jules Verne}
\\ {Facult\'e des Sciences}
\\ {33, rue Saint-Leu}
\\ {80039 Amiens CEDEX 1, France}}
\email{alberto.farina@u-picardie.fr}

\author{Enrico Valdinoci}
\address{ {Weierstra{\ss} Institut f\"ur Angewandte Analysis und
Stochastik} \\
{Mohrenstrasse, 39} \\
{10117 Berlin, Germany} }
\email{enrico.valdinoci@wias-berlin.de}

\thanks{It is a pleasure to thank
Yoshihiro Tonegawa for a very instructive
discussion held in Banff. This work has been
supported by the ERC grant 277749 ``EPSILON Elliptic
Pde's and Symmetry of Interfaces and Layers for Odd Nonlinearities''.}

\begin{abstract}
We study bounded, monotone solutions of~$\Delta u=W'(u)$
in the whole of~$\R^n$,
where~$W$ is a double-well potential. We prove that
under suitable assumptions on the limit interface
and on the energy growth, $u$ is $1$D.

In particular, differently from the previous literature, the solution is not
assumed to have minimal properties and the cases studied lie outside
the range of $\Gamma$-convergence methods.

We think that this approach could be fruitful in concrete situations,
where one can observe the phase separation at a large scale and whishes
to deduce the values of the state parameter in the vicinity of the interface.

As a simple example of the results obtained with this point of view,
we mention that monotone solutions with energy bounds,
whose limit interface does not contain a vertical line through the origin,
are $1$D, at least up to dimension~$4$.
\end{abstract}

\maketitle
\tableofcontents

\section*{Notation}

We take~$n\ge2$. A point~$x\in\R^n$ will be often written
as~$x=(x',x_n)\in\R^{n-1}\times\R$.

For any~$1\le i\le n$, the partial derivatives
with respect to~$x_i$ will be denoted by
$$\partial_i=\partial_{x_i}=\frac{\partial}{\partial x_i}.$$
Also, given an ambient domain~$\Omega\subseteq\R^n$
and a set~$F\subseteq\Omega$, we denote 
by~$\chi_F$ its characteristic function, i.e.
$$ \chi_F(x):=\left\{\begin{matrix}
1 & {\mbox{ if $x\in F$,}}
\\ 0& {\mbox{ if $x\in {\mathcal{C}}F$,}}
\end{matrix}\right.$$
where~${\mathcal{C}}F:=\Omega\setminus F$ is the complement of~$F$
in its ambient space.

We denote by~$W$ the classical ``double-well potential''~$W(r)=(1-r^2)^2$
(more general type of double-well potentials may be treated
in the same way), and we will study solutions
\begin{equation}\label{PDE}
\begin{split}&u\in C^2(\R^n,[-1,1]) {\mbox{ of }} \\&\Delta u(x)=W'(u(x))
{\mbox{ for any }}
x\in\R^n\end{split}\end{equation}
under the monotonicity condition
\begin{equation}\label{mono}
{\frac{\partial u}{\partial x_n}}(x)\,>\,0
{\mbox{ for any }}
x\in\R^n.
\end{equation}
After~\cite{DG}, condition~\eqref{mono} has become
classical in the study of semilinear equations. {F}rom the
variational point of view, it implies that
\begin{equation}\label{STA}
{\mbox{$u$ is a stable solution,}} \end{equation}
i.e. the second variation of the energy
is nonnegative, see e.g. Corollary 4.3 in \cite{AC}.

For such~$u$ and any~$\eps>0$, we define the rescaled solution
$$ u_\eps(x):=u(x/\eps).$$
We observe that, by the Maximum Principle
and~\eqref{mono}, we have that~$|u(x)|<1$ for every~$x\in\R^n$ and so,
in particular,
\begin{equation}\label{EMU}
u_\eps(0)=u(0)\in (-1,1).
\end{equation}
Given a bounded open set~$\Omega\subset\R^n$, we also
consider the energy functional associated to~\eqref{PDE},
namely
$$ {\mathcal{E}}(u,\Omega):=\int_\Omega\frac{|\nabla u(x)|^2}{2}
+W(u(x))\,dx.$$
We say that~$u$ is a (local) minimizer if, for any
bounded open set~$\Omega$ and any~$\varphi\in C^\infty_0(\Omega)$,
we have that
$$ {\mathcal{E}}(u,\Omega)\le {\mathcal{E}}(u+\varphi,\Omega).$$
Similarly, one says that~$u$ is quasiminimal if, for some~$Q\ge1$, one has that
$$ {\mathcal{E}}(u,\Omega)\le Q\,{\mathcal{E}}(u+\varphi,\Omega)$$
for any bounded open set~$\Omega$ and any~$\varphi\in C^\infty_0(\Omega)$.

\section{Introduction}

The study of the PDE in \eqref{PDE}
under the monotonicity assumption \eqref{mono}
is a classical topic in semilinear elliptic equations 
and it goes back, at least, to the study
of the Ginzburg-Landau-Allen-Cahn
phase segregation model, in connection with
the theory of hypersurfaces with minimal perimeter.
In particular,
the following striking problem
was posed in~\cite{DG}:

\begin{question}\label{GQ}
Let~$u$ be a solution of~\eqref{PDE}
satisfying~\eqref{mono}.
Is it true that~$u$ is~$1$D, i.e. that
all the level sets of~$u$ are hyperplanes, at least if $n\le8$?
\end{question}

We refer to~\cite{GG,AC} and also~\cite{BCN,AAC}
for the proof that Question~\ref{GQ}
has a positive answer in dimension~$2\le n\le 3$.
See also~\cite{S}, where it is shown that
Question~\ref{GQ}
also has a positive answer in dimension~$4\le n\le8$
provided that
\begin{equation}\label{LI}
\lim_{x_n\rightarrow-\infty} u(x',x_n)=-1
\quad {\mbox{ and }} \quad
\lim_{x_n\rightarrow+\infty} u(x',x_n)=1.\end{equation}
When~$n\ge9$, an example of a solution~$u$
satisfying~\eqref{PDE}, \eqref{mono} and \eqref{LI} that is not~$1$D
was constructed in~\cite{delpinokowei}, showing that the dimensional
constraint in Question~\ref{GQ} cannot be removed.

See also \cite{Tran} for some symmetry results that hold under conditions
at infinity which are weaker than~\eqref{LI}.
In particular, it is shown in~\cite{Tran} that condition~\eqref{LI}
can be relaxed to suitable symmetry assumptions on the asymptotic
profiles
\begin{equation}\label{profiles}\begin{split}
&\overline u(x'):=\lim_{x_n\rightarrow-\infty} u(x',x_n), 
\\ 
&\underline u(x'):=\lim_{x_n\rightarrow+\infty} u(x',x_n)
..\end{split}\end{equation}
More precisely, in~\cite{Tran}
it is proved that solutions of~\eqref{PDE}
satisfying~\eqref{mono} are $1$D if:
\begin{itemize}
\item $2\le n\le 4$ and at least one between~$\overline u$
and~$\underline u$ are $2$D,
\item $2\le n\le 8$ and both~$\overline u$
and~$\underline u$ are $2$D.
\end{itemize}
It is also proved in~\cite{Tran} that
solutions of~\eqref{PDE} that satisfy~\eqref{mono} 
are $1$D if~$2\le n\le 8$, provided that
at least one level set is a complete graph.

Other symmetry results are obtained in~\cite{Tran}
for quasilinear equations and for quasiminimal solutions.

We also refer to~\cite{LE} for further motivation and
a review on Question~\ref{GQ}, and to~\cite{ZAMP}
for its connection with an important problem posed
by~\cite{Bangert}.

\medskip

In the light of the above mentioned results, we
have that Question~\ref{GQ} is still open in dimension~$4\le n\le8$, and our
paper would like to be a further step towards this
direction.\medskip

In the literature, most of the research related to
Question~\ref{GQ} heavily relies on the analysis of minimizers
of the energy functional. 
This approach was also inspired by the classical $\Gamma$-convergence
results (see e.g.~\cite{Modicark}), 
which established an important link
between the level sets of the solutions and the study of hypersurfaces
with minimal perimeter.

On the other hand, from the point of view
of pure mathematics, it is interesting to consider the case
of solutions that do not necessarily
have minimizing properties, 
or whose minimizing properties are not completely known a-priori.
As a matter of fact, these solutions also appear in concrete
situations, since there is numerical evidence that
some solutions show unstable patterns before settling down
to more stable configurations (see e.g. \cite{MR950604}).

Therefore, the scope of this paper is to investigate symmetry results
without assuming any minimality condition on the solution,
but only monotonicity assumptions, energy bounds and some
geometric information on the limit interface.
\medskip

For this, we will derive rigidity and symmetry results
from either the behavior of the limit level set
or the one of the limit varifold, 
using some results in~\cite{HT}
in order to introduce and describe the limit interface
without minimizing assumptions (notice that in this case
the $\Gamma$-convergence theory cannot be applied).
\medskip

More precisely, the cornerstone to describe
the limit interface lies in the assumption that
$u$ satisfies the following energy bound: for any~$R>1$,
\begin{equation}\label{e bound}
{\mathcal{E}}(u,B_R)\le CR^{n-1},
\end{equation}
for some~$C>0$ independent of $R$.
Such energy bound is a classical assumption in the
setting of semilinear elliptic equations
(see e.g.~\cite{HT})
and it is satisfied by minimizers and
monotone solutions with some conditions at infinity
(see e.g. Theorem 5.2 in \cite{AC}). As a matter of fact, it
is valid for quasiminimal solutions too (see Lemma~10 in~\cite{FV})
and it is
also implied by the weaker condition on the potential energy
$$ \int_{B_R} W(u(x))\,dx\le CR^{n-1},$$
see~\cite{Mod}.
By scaling, the energy bound in~\eqref{e bound}
implies that the rescaled energy with density
$$ \frac{\varepsilon |\nabla u_\varepsilon(x)|^2}{2}
+\frac{W(u_\varepsilon(x))}{\varepsilon}\,dx$$
is locally bounded uniformly in~$\varepsilon>0$. Under this
condition, and recalling~\eqref{STA},
the results of~\cite{HT} come into play.
First of all,
we fix a domain (for convenience a cylinder)
and we look at the asymptotics of the rescaled level sets
in such domain. That is, given~$d>0$ and~$h>0$ we denote by~${\rm C}(d,h)$
the (open)
cylinder of base radius~$d$ and height~$2h$, i.e. we set
\begin{equation}\label{CYL}
{\rm C}(d,h):= \{ x\in\R^n {\mbox{ s.t. }} |x'|<d
{\mbox{ and }} |x_n|<h\}.\end{equation}
Then we fix (once and for all in this paper) the quantities~$d_o$, $h_o>0$ and
(see e.g. page~52 in~\cite{HT}) we have that,
up to a subsequence, $u_\varepsilon$ converges a.e. in~${\rm C}(d_o,h_o)$
to $\pm1$, i.e.
we can define~$\LLM$ as the set of points~$p\in {\rm C}(d_o,h_o)$
such that~$u_\eps(p)\to-1$, and then~$u_\eps$ 
approaches a.e. the step function~$\chi_{{\mathcal{C}}\LLM}-\chi_\LLM$.
\medskip

We observe that~$\chi_{{\mathcal{C}}\LLM}-\chi_\LLM$ is a measurable
function, since it is obtained by limit of measurable (and, in fact,
continuous) functions, and therefore~$\LLM$ is a measurable set.
Nevertheless, it is worth to point out that we
do not identify~$\LLM$ with the set of
its points of Lebesgue density~$1$, since we have defined it
directly via the pointwise convergence of the rescaled solution~$u_\eps$.
\medskip

It is also interesting to notice that
\begin{equation}\label{ECE0}
0\in {\mathcal{C}}\LLM,
\end{equation}
thanks to~\eqref{EMU}.

In a symmetric way, we also define~$\LLP$ as
the set of points~$p\in {\rm C}(d_o,h_o)$
such that~$u_\eps(p)\to 1$.
\medskip

Having clearly stated the notion of the limit level set,
now we introduce a more general object
which encodes further asymptotic and geometric properties of
the solution. Namely
(see Theorem~1 and Proposition 4.2. of~\cite{HT})
we have that
there exists a varifold~$V$ (which we will call ``limit varifold''
and whose geometric support will be denoted by~$V$
as well) such that (up to subsequences):
\begin{itemize}
\item $u_\varepsilon\to\pm1$ uniformly on each connected compact subset of 
${\rm C}(d_o,h_o)\setminus V$;
\item for any $ {\tilde U} \Subset  {\rm C}(d_o,h_o)$ and for any~$c\in(-1,1)$ the set~$\{|u|\le c\} \cap {\tilde U} $  converges uniformly to~$ V \cap {\tilde U}$, 
i.e., if we set
\begin{equation} \label{6bis}
V_\delta:= \bigcup_{p\in V} B_\delta(p),\end{equation}
then for any~$\delta>0$ there exists~$\varepsilon_o>0$
such that for any~$\varepsilon\in(0,\varepsilon_o)$ we have that
\begin{equation}\label{LU0} 
\{|u_\varepsilon|\le c\}\cap {\tilde U} \subseteq V_\delta \cap {\tilde U}.\end{equation}
\end{itemize}
\medskip

We remark that the limit set~$\LLM$ and the limit varifold~$V$
that we have discussed here have a simple
geometric explanation in terms of the physical interpretation
of the Allen-Cahn equation: namely, 
they represent the limit
interface that separate two coexisting phase states.
\medskip

Using the setting mentioned above,
the first result that we present deals with the rigidity properties
of the solutions which are inherited from the structure of the limit level set~$\LLM$:

\begin{theorem}\label{T1.0}
Let~$u$ be a solution of~\eqref{PDE}
satisfying~\eqref{mono} and~\eqref{e bound}.
Suppose that there exist~$d\in(0,d_o]$, and~$\overline K$, $K\in(0,h_o]$, with~$\overline K>K$,
such that
\begin{equation}\label{ECE}
\LLM\cap {\rm C}(d,\overline K)
\ne\varnothing, \qquad \LLP\cap {\rm C}(d,\overline K)
\ne\varnothing,
\end{equation}
and
\begin{equation}\label{TRAP}
\Big((\partial \LLM)\cup(\partial \LLP)\Big)\cap {\rm C}(d,\overline K)
\subseteq \{|x_n|< K\}.
\end{equation}
Then:
\begin{itemize}
\item[{(i)}] The limits in~\eqref{LI} hold true,
\item[{(ii)}] The solution~$u$ is a local
minimizer,
\item[{(iii)}] If $2\le n\le 8$, then~$u$ is~$1$D.
\end{itemize}
\end{theorem}

We observe that condition~\eqref{ECE}
is quite natural, since it is consistent with the physical
interpretation of the model describing the coexistence
of two phases separated by an interface. Moreover, it
is always satisfied if the solution~$u$ is minimal
or, more generally, quasiminimal: indeed, in this case,
it follows from~\eqref{EMU}
and Corollary~13 in~\cite{FV} that~$\min\{|B_\delta\cap \LLM|,\,|B_\delta\cap\LLP|\}\ge c\delta^n$, for a suitable~$c>0$.\medskip

We also point out that
condition~\eqref{ECE} is also directly implied by~\eqref{mono}
and suitable geometric constraints on the zero level sets of the solution
(for instance, by the condition that~$\{u=0\}\subseteq \{|x_n|\le\kappa\}$
for some~$\kappa\ge0$).
\medskip

The geometric restriction on the limit level set~$\LLM$ in~\eqref{TRAP}
is depicted in Figure~1
(the picture also remarks that, in principle, condition \eqref{TRAP}
allows~$\partial \LLM$ to have vertical parts outside the origin, though,
of course, some further restrictions on~$\LLM$ are imposed by~\eqref{mono}).
\bigskip

\begin{center} 
\includegraphics[height=2in]{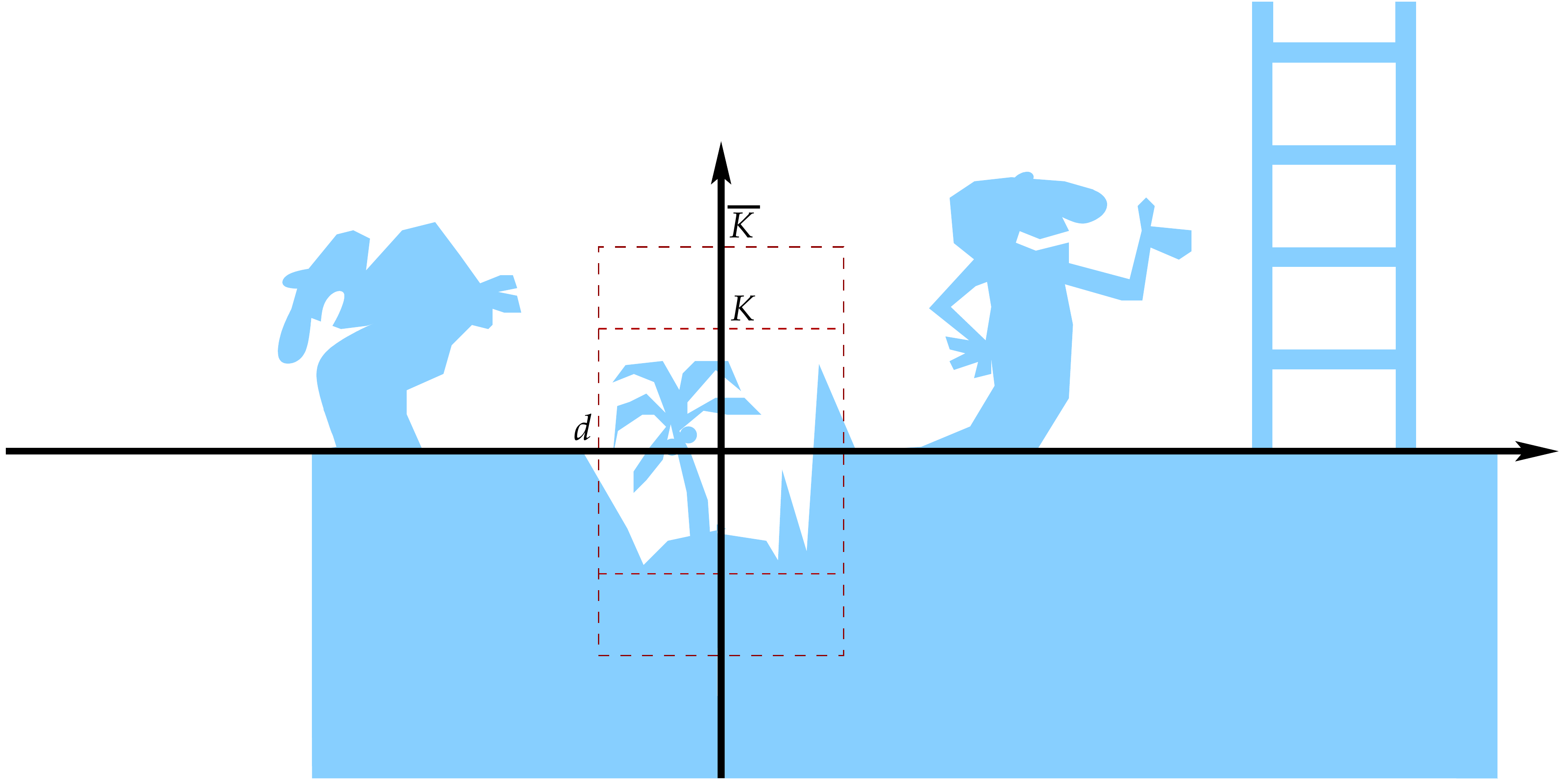}
\end{center}
\nopagebreak\centerline{\footnotesize\em
Figure~1. The trapping condition on the set~$\LLM$ (which is the colored
region) given by~\eqref{TRAP}.}
\bigskip

As a matter of fact, we will derive Theorem~\ref{T1.0}
from the following more general result, in which
the asymptotic profiles of the solutions are determined only by
a one-side constraint on the limit level set:

\begin{theorem}\label{T04}
Let~$u$ be a solution of~\eqref{PDE}
satisfying~\eqref{mono} and~\eqref{e bound}.
Suppose that there exist~$d\in(0,d_o]$, 
and~$\overline K$, $K\in(0,h_o]$, with~$\overline K>K$,
such that
\begin{equation}\label{ECE00}
\LLM\cap {\rm C}(d,\overline K)
\ne\varnothing
\end{equation}
and
\begin{equation}\label{TRAP00}
(\partial \LLM)\cap {\rm C}(d,\overline K)
\subseteq \{x_n> -K\}.
\end{equation}
Then
\begin{equation}\label{TRAP01}
\lim_{x_n\to-\infty}u(x',x_n)=-1.\end{equation}
Moreover, if~$2\le n\le 4$, then~$u$ is~$1$D.
\medskip

Similarly, if
\begin{equation}\label{ECE00-bis}
\LLP\cap {\rm C}(d,\overline K)
\ne\varnothing
\end{equation}
and
\begin{equation}\label{TRAP00-bis}
(\partial \LLP)\cap {\rm C}(d,\overline K)
\subseteq \{x_n< K\},
\end{equation}
then
$$ \lim_{x_n\to+\infty}u(x',x_n)=1.$$
Moreover, if~$2\le n\le 4$, then~$u$ is~$1$D.
\end{theorem}

It is worth to point out that condition~\eqref{ECE00}
is quite natural: indeed, we already know from~\eqref{ECE0}
that~${\mathcal{C}}\LLM\cap {\rm C}(d,\overline K)\ne\varnothing$,
thus condition~\eqref{ECE00}
may be seen as the symmetric counterpart of it.
Also, by~\eqref{ECE0}, we have that~\eqref{ECE00}
is equivalent to
$$(\partial \LLM)\cap {\rm C}(d,\overline K)\ne\varnothing.$$
\medskip

We also observe that Theorem~\ref{T04} (and so Theorem~\ref{T1.0})
possesses
a measure theoretic version, in which one is allowed to modify~$\LLM$ or~$\LLP$
by sets of measure zero. We give a detailed statement for completeness:

\begin{theorem}\label{T04-BIS}
Let~$u$ be a solution of~\eqref{PDE}
satisfying~\eqref{mono} and~\eqref{e bound}.
Suppose that there exist~$d\in(0,d_o]$, 
and~$\overline K$, $K\in(0,h_o]$, with~$\overline K>K$,
and a set~$\EEEM\subseteq \LLM$ such 
that
\begin{equation}\label{ZERO}
|\LLM\setminus\EEEM|=0
,\end{equation}
\begin{equation}\label{ECE00-BIS}
|\EEEM\cap {\rm C}(d,\overline K)
|>0
\end{equation}
and
\begin{equation}\label{TRAP00-BIS}
(\partial \EEEM)\cap {\rm C}(d,\overline K)
\subseteq \{x_n> -K\}.
\end{equation}
Then
\begin{equation}\label{8fd67sgh}
\lim_{x_n\to-\infty}u(x',x_n)=-1.\end{equation}
Moreover, if~$2\le n\le 4$, then~$u$ is~$1$D.
\medskip

Similarly, if there exists
a set~$\EEEP\subseteq \LLP$ such
that
\begin{equation}\label{ZERO.1}
|\LLP\setminus\EEEP|=0,\end{equation}
\begin{equation}\label{ECE00-BIS.1}|\EEEP\cap {\rm C}(d,\overline K)|>0\end{equation}
and
\begin{equation}\label{TRAP00-BIS.1}
(\partial \EEEP)\cap {\rm C}(d,\overline K)
\subseteq \{x_n< K\},\end{equation}
then
$$ \lim_{x_n\to+\infty}u(x',x_n)=1.$$
Moreover, if~$2\le n\le 4$, then~$u$ is~$1$D.\medskip

In particular, if there are sets~$\EEEM\subseteq\LLM$ and~$\EEEP\subseteq\LLP$
satisfying~\eqref{ZERO},
\eqref{ECE00-BIS}, \eqref{TRAP00-BIS},
\eqref{ZERO.1},
\eqref{ECE00-BIS.1} and~\eqref{TRAP00-BIS.1}, then
\begin{itemize}
\item[{(i)}] The limits in~\eqref{LI} hold true,
\item[{(ii)}] The solution~$u$ is a local
minimizer,
\item[{(iii)}] If $2\le n\le 8$, then~$u$ is~$1$D.
\end{itemize}
\end{theorem}

The reader may compare~\eqref{ECE00-BIS} with~\eqref{ECE00} and~\eqref{TRAP00-BIS}
with~\eqref{TRAP00}.
\medskip


Next we use the approach that we developed in \cite{Tran}, in which symmetry results and qualitative properties of the solutions to the considered problem are obtained when at least one level set of the solution is a complete graph. More precisely, we have the following results:

\begin{theorem}\label{T1.0vecchio}
Let~$u$ be a solution of~\eqref{PDE}
satisfying~\eqref{mono} and~\eqref{e bound}.
Suppose that there exist~$d\in(0,d_o]$ and 
$\overline K\in(0,h_o]$ such that
\begin{equation}\label{lanuova}
\LLM\cap {\rm C}(d,\overline K)\ne\varnothing
\ {\mbox{ and }} \
\LLP\cap {\rm C}(d,\overline K)\ne\varnothing.
\end{equation}
Suppose also that there exist~$K\in(0,\overline K)$,
a set~$\EEE\subseteq
{\rm C}(d_0,h_0)$ and a value~$c\in(-1,1)$ such that
\begin{equation}\label{aCCO}
{\mbox{the level set~$\{u_\eps=c\}\cap {\rm C}(d,\overline K)$
converges uniformly to~$\partial\EEE$}}
\end{equation}
and
\begin{equation}\label{TRAPvecchio}
(\partial \EEE)\cap {\rm C}(d,\overline K)
\subseteq \{|x_n|< K\}.
\end{equation}
Then:
\begin{itemize}
\item[{(i)}] The level set $ \{u = c\}$ is a complete graph
in the vertical direction (i.e., for any~$x'\in\R^{n-1}$ there exists a unique $x_n(x')$ such that~$u(x',x_n(x'))=c$),
\item[{(ii)}] The limits in~\eqref{LI} hold true,
\item[{(iii)}] The solution~$u$ is a local
minimizer,
\item[{(iv)}] If $2\le n\le 8$, then~$u$ is~$1$D.
\end{itemize}
\end{theorem}

We observe that condition~\eqref{aCCO} is somehow natural.  
For instance,  if~$u$ is a quasiminimal solution, condition~\eqref{aCCO} is satisfied by choosing a set $\EEE$ which differs from $\LLM$ by a set of measure zero (cf. Corollary~2 in~\cite{FV}). The same remark applies to the next Theorem \ref{T1.0vecchio-nuova}

\medskip

A variant of Theorem~\ref{T1.0vecchio}
that replaces assumption~\eqref{lanuova}
with some measure theoretic conditions on the set~$\EEE$
goes as follows:

\begin{theorem}\label{T1.0vecchio-nuova}
Let~$u$ be a solution of~\eqref{PDE}
satisfying~\eqref{mono} and~\eqref{e bound}.
Suppose that there exist~$d\in(0,d_o]$, 
$\overline K$, $K\in(0,h_o]$, with~$\overline K>K$, and
a set~$\EEE\subseteq
{\rm C}(d_0,h_0)$ such that
\begin{equation}\label{COC}
|\EEE\setminus\LLM|=|\LLM\setminus \EEE|=0,
\end{equation}
\begin{equation}\label{ECEvecchio}
|\EEE\cap {\rm C}(d,\overline K)|>0, \qquad
|({\mathcal{C}}\EEE)\cap {\rm C}(d,\overline K)|>0,
\end{equation}
for some~$c\in(-1,1)$
\begin{equation*}
{\mbox{the level set~$\{u_\eps=c\}\cap {\rm C}(d,\overline K)$
converges uniformly to~$\partial\EEE$,}}
\end{equation*}
and
\begin{equation}\label{ECEvecchio2}
(\partial \EEE)\cap {\rm C}(d,\overline K)
\subseteq \{|x_n|< K\}.
\end{equation}
Then:
\begin{itemize}
\item[{(i)}] The level set~$ \{u = c\}$ is a complete graph
in the vertical direction,
\item[{(ii)}] The limits in~\eqref{LI} hold true,
\item[{(iii)}] The solution~$u$ is a local minimizer,
\item[{(iv)}] If $2\le n\le 8$, then~$u$ is~$1$D.
\end{itemize}
\end{theorem}

We notice that the difference between Theorems~\ref{T1.0vecchio}
and~\ref{T1.0vecchio-nuova}
is that assumption~\eqref{lanuova},
which involves the sets~$\LLM$ and~$\LLP$,
is replaced by assumptions~\eqref{COC}
and~\eqref{ECEvecchio}, which are similar but only involve the
set~$\EEE$. Also, in Theorem~\ref{T1.0vecchio},
one does not need to assume a-priori that the set~$\EEE$
coincides with~$\LLM$ up to sets of measure zero.

We also remark that assumption~\eqref{COC} may be replaced
by the similar one that involves~$\LLP$ instead of~$\LLM$, namely
$$|\EEE\setminus\LLP|=|\LLP\setminus \EEE|=0.$$
\medskip

Now we present further
rigidity and symmetry results that follow, at least
in dimension~$4$, under the
assumption that the level set~$\{u=0\}$ is confined below
(or above) a complete graph. We emphasize that no energy assumption is needed in this case. Indeed (6), in this case, is a byproduct of the other hypotheses. Namely, we have the following result:

\begin{theorem}\label{TTX}
Let~$u$ be a solution of~\eqref{PDE}
satisfying~\eqref{mono}.

Suppose that~$\{u=0\}$ is confined below
a complete graph, i.e. suppose that there exists~$\gamma\in C(\R^{n-1})$
such that
\begin{equation}\label{02}
\{u=0\}\subseteq \{ x_n \le \gamma(x')\}.\end{equation}
Then
\begin{equation}\label{Tx1} \lim_{x_n\rightarrow+\infty} u(x',x_n)=1.\end{equation}
Similarly, if~$\{u=0\}$ is confined above
a complete graph, then
$$ \lim_{x_n\rightarrow-\infty} u(x',x_n)=-1.$$

In any case, the energy bound in~\eqref{e bound} holds true.

Moreover, if~$2\le n\le 4$, then~$u$ is~$1$D.
\end{theorem}

While in Theorem~\ref{T1.0}, \ref{T04}, \ref{T04-BIS} and~\ref{TTX} we have deduced
symmetry results from the structure of
the limit level set~$\LLM$, now
we will prove further rigidity results
under some geometric control on the limit varifold~$V$:

\begin{theorem}\label{T1}
Let~$u$ be a solution of~\eqref{PDE}
satisfying~\eqref{mono} and~\eqref{e bound}
and let~$V$ be the associated limit varifold.

Suppose that there exist two points~$\overline x=(0,\dots,0,\overline x_n)$
and~$\underline x=(0,\dots,0,\underline x_n)$ that do not belong to~$V$,
with
$$h_o>\overline x_n>0>\underline x_n>-h_o.$$
Then:
\begin{itemize}
\item[{(i)}] The limits in~\eqref{LI} hold true,
\item[{(ii)}] The solution~$u$ is a local
minimizer,
\item[{(iii)}] If $2\le n\le 8$, then~$u$ is~$1$D.
\end{itemize}
\end{theorem}

Concerning the assumptions of Theorem~\ref{T1.0} and~\ref{T1},
it is worth to remark that, in general, condition \eqref{mono}
is not enough to imply~\eqref{TRAP}, nor the existence
of the points~$\overline x$ and~$\underline x$ in
Theorem~\ref{T1}.
The reason is, roughly speaking, that the limit interface
could be ``vertical''. For instance, let
$$ u(x_1,x_2):=-\frac{2}{\pi}\,{\rm arctg}\, ( x_1-\,{\rm arctg}\, x_2).$$
Since the derivative of the function~$t\mapsto \,{\rm arctg}\, t$ is strictly
positive, one readily checks that $\partial_{x_2} u>0$, hence
condition~\eqref{mono} is satisfied.
Nevertheless, we have that
$$ u_\varepsilon(x_1,x_2)=
-\frac{2}{\pi}\,{\rm arctg}\, \left( \frac{x_1}{\varepsilon}
-\,{\rm arctg}\, \frac{x_2}{\varepsilon}\right),$$
therefore, for any $c\in(-1,1)$,
$$ \{u_\varepsilon =c\}= 
\left\{ x_1= \varepsilon\left[\,{\rm arctg}\, \frac{x_2}{\varepsilon}
-\,{\rm tg}\, \frac{c\,\pi}{2}\right]\right\}.$$
Accordingly, $\{u_\varepsilon=c\}$ approaches $\{x_2=0\}$,
which shows that the geometric condition in Theorem~\ref{T1}
is not satisfied in this case.
\medskip

In dimension~$n\le4$, Theorem~\ref{T1} can be 
strengthened:
namely, it is enough that the limit varifold does not contain the vertical
line in order to deduce symmetry properties, as stated in next result:

\begin{theorem}\label{T2}
Let~$2\le n\le 4$ and~$u$ be a solution of~\eqref{PDE}
satisfying~\eqref{mono} and~\eqref{e bound}
and let~$V$ be the associated limit varifold.

Suppose that there is a point of the vertical line ${r}_\star:=\{ (0,\dots,0,t),\;t \in \R \}$ which 
is not contained in~$V$.

Then~$u$ is~$1$D.
\end{theorem}

We remark that the results obtained are valid also for more
general bistable nonlinearities than the classical Allen-Cahn
equation (as a matter of fact,
only in Theorem~\ref{TTX}
one needs the nonlinearity to grow linearly at the origin,
in order to exploit the results of~\cite{Far99,FAR}).

Also, it is worth to point out that Theorems~\ref{T1.0}, \ref{T04}, \ref{T04-BIS}
and~\ref{TTX} do not use the limit varifold~$V$, but
only the limit level sets~$\LLM$ and $\LLP$. 
\medskip

We also observe that the assumptions used in this paper
are geometric and related to the asymptotic behavior
of the interface of the problem: we think that these
kind of hypotheses are somehow natural from the viewpoint
of the physical applications and they may be easier to check
in concrete cases than the usual variational assumptions
(such as the one requiring minimality of the solution).

Furthermore, conditions on the limit interface are perhaps more feasible
to be checked in concrete applications, when one ``sees'' in practice
the interface, and then can deduce from the symmetry results of the theory
that the phase state depends, at a large scale, only on the distance
from the phase separation.
\medskip

In this, spirit, as a final observation, we point out
that the symmetry of the solution is, in the end, somehow equivalent
to the flatness of its interface, according to the following result:

\begin{corollary}\label{CORO}
Let~$u$ be a solution of~\eqref{PDE}
satisfying~\eqref{mono} and~\eqref{e bound}. 
Then~$u$ is~$1$D
if and only if~$\partial \LLM=\partial\LLP$ is a non-vertical hyperplane containing the origin.
\end{corollary}

The rest of the paper is organized as follows.
In Section~\ref{GA}
we prove the rigidity results based on the geometric behavior
of the limit set~$\LLM$ (namely, one after the other, Theorems~\ref{T04}, 
\ref{T1.0}, \ref{T04-BIS}, \ref{TTX},
\ref{T1.0vecchio} and~\ref{T1.0vecchio-nuova},
as well as Corollary~\ref{CORO}, which follows from Theorem~\ref{T1.0}).
Then, in Section~\ref{AG} 
we deal with the rigidity results coming from the geometric behavior
of the limit varifold~$V$ (namely Theorems~\ref{T1} and~\ref{T2}).
More precisely, in Subsection~\ref{GAG}
we describe the influence of the limit varifold
on the asymptotic behavior of the solution,
and  Theorems~\ref{T1} and \ref{T2} 
will be proved in Subsection~\ref{TGT}.

\section{Rigidity and symmetry from the limit level set}\label{GA}

Here we prove the rigidity results that rely on the geometric structure
of the limit sets~$\LLM$ and $\LLP$ (namely,
Theorems~\ref{T04}, \ref{T1.0}, \ref{T04-BIS} and~\ref{TTX}).

\begin{proof}[Proof of Theorem \ref{T04}.]
We prove only the first part of Theorem \ref{T04}
since the second part follows from the first one by replacing~$u=u(x)$
with~$ -u(x',-x_n) $.

We start by showing that, if~\eqref{ECE00} and~\eqref{TRAP00}
holds true, then the limit in~\eqref{TRAP01} holds true as well.
For this scope,
we let
\begin{equation}\label{P13Q}\begin{split}& Q_+:=\{|x'|< d\}\times \{ x_n\in [K,\overline K)\}\\
{\mbox{and }}& Q_-:=\{|x'| < d\}\times \{ x_n\in (-\overline K,-K]\}.\end{split}\end{equation}
Notice that
\begin{equation}\label{CP} {\rm C}(d,\overline K)=Q_-\cup Q_+\cup {\rm C}(d,K).\end{equation}
We claim that
\begin{equation}\label{ASS1}
{\mbox{either $\ Q_-\cap \LLM=\varnothing \ $ or
$\ Q_-\setminus \LLM=\varnothing \ $.}}
\end{equation}
To prove it, we notice that
if~\eqref{ASS1} were false,
the set ~$Q_-$ would have to contain a point
of the boundary of~$\LLM$.
This gives
a contradiction with~\eqref{TRAP00} and so we have
proved~\eqref{ASS1}.

Now we improve~\eqref{ASS1} by showing that
\begin{equation}\label{SASS}
Q_-\subseteq \LLM.
\end{equation}
Suppose not. Then~\eqref{ASS1} implies that~$Q_-\cap \LLM=\varnothing$,
hence for any~$x\in Q_-$ we have that~$
u_\varepsilon(x)\not\to-1$ as~$\eps\to0^+$. Then, fix any point~$y\in {\rm C}(d,\overline K)$
and recall~\eqref{CP} to find a point~$\tilde y\in Q_-$ such that~$\tilde y'=y'$
and~$\tilde y_n\le y_n$. {F}rom~\eqref{mono}, we know that~$\partial_{x_n}u_\eps>0$.
By collecting these pieces of information, we see that
$$ u_\eps(y)\ge u_\eps(\tilde y)\not\to-1$$
as~$\eps\to0^+$. In particular we have that~$u_\eps(y)\not\to-1$, that is~$y\not\in \LLM$.
Since this is valid for any~$y\in {\rm C}(d,\overline K)$, we have shown that~${\rm C}(d,\overline K)
\cap \LLM=\varnothing$. This is in contradiction with~\eqref{ECE00}
and so the proof of~\eqref{SASS} is complete.

Now, we set~$\vartheta:=(K+\overline K)/2$ and~$P:=(0,\dots,0,-\vartheta)$.
We have that~$P\in Q_-$, thus,
by~\eqref{SASS}, we conclude that~$P\in \LLM$ and so
$$ -1=\lim_{\eps\to0^+}u_\eps(P)=
\lim_{\eps\to0^+}u \left(0,\dots,0,-\frac{\vartheta}{\eps}\right)=\underline u(0).$$
This and the Maximum Principle implies that~$\underline u$
is constantly equal to~$-1$, which establishes~\eqref{TRAP01}.

Then, if additionally~$2\le n\le4$, using~\eqref{TRAP01}
and Theorem~1.2 of~\cite{Tran}, we obtain that~$u$ is~$1$D.
\end{proof}

\begin{proof}[Proof of Theorem \ref{T1.0}.]
We remark that~\eqref{ECE}
implies both~\eqref{ECE00} and \eqref{ECE00-bis}. Moreover, condition~\eqref{TRAP}
implies both~\eqref{TRAP00}
and~\eqref{TRAP00-bis}. Accordingly, we can use Theorem~\ref{T04}
and obtain claim~(i) of Theorem \ref{T1.0}.
Then, claims~(ii) and~(iii) follow from
Theorem~1.3 of~\cite{Tran}.\end{proof}

\begin{proof}[Proof of Corollary \ref{CORO}.]
Suppose that~$u$ is~$1$D. Then there exist~$\omega\in {\rm S}^{n-1}$
and a function of one variable~$u_o$ such that~$u(x)=u_o(\omega\cdot x)$
for every~$x\in\R^n$. By solving the ODE satisfied by~$u_o$ we obtain that
$$ \lim_{t\to\pm\infty}u_o(t)=\pm1.$$
We claim that
\begin{equation}\label{stra}
\LLM=\{\omega\cdot x<0\} \ {\mbox{ and }} \
\LLP=\{\omega\cdot x>0\}.
\end{equation}
Indeed, 
$$\lim_{\eps\to0^+}u_\eps(x)=
\lim_{\eps\to0^+} u\left(\frac{x}{\eps}\right)=
\lim_{\eps\to0^+}u_o\left(
\frac{\omega\cdot x}{\eps}\right)=\left\{
\begin{matrix}
1 & {\mbox{ if }}\omega\cdot x>0,\\
u_o(0) & {\mbox{ if }}\omega\cdot x=0,\\
-1 & {\mbox{ if }}\omega\cdot x<0.
\end{matrix}
\right. $$
This and~\eqref{EMU}
prove~\eqref{stra}. {F}rom~\eqref{stra},
it follows that~$\partial \LLM=\partial\LLP=\{\omega\cdot x=0\}$,
which is a non-vertical hyperplane containing the origin.\medskip

Viceversa, let us now suppose that~$\partial \LLM=\partial\LLP$ is
a non-vertical hyperplane containing the origin.
Then conditions~\eqref{ECE} and~\eqref{TRAP} are satisfied and we
infer from Theorem~\ref{T1.0} that~$u$ is a minimizer. In particular,
we can use Corollary~7 in~\cite{FV} and obtain that~$u$ is~$1$D,
without any restriction on the dimension of the ambient space.
\end{proof}

\begin{proof}[Proof of Theorem \ref{T04-BIS}.]
The proof is a modification of the one Theorem \ref{T04},
according to these lines. First, one replaces~$\LLM$ with~$\EEEM$
in~\eqref{ASS1}, i.e. instead of~\eqref{ASS1} one proves that
\begin{equation}\label{ASS1-NUOVA}
{\mbox{either $\ Q_-\cap \EEEM=\varnothing \ $ or
$\ Q_-\setminus \EEEM=\varnothing \ $.}}
\end{equation}
This follows easily from~\eqref{TRAP00-BIS}. Then, one replaces~$\LLM$
with~$\EEEM$ in \eqref{SASS}, i.e. one shows that
\begin{equation}\label{SASS-NUOVA}
Q_-\subseteq \EEEM.
\end{equation}
The proof is similar to the one of~\eqref{SASS}, but it makes use
of~\eqref{ZERO} and~\eqref{ECE00-BIS}. Namely, if~\eqref{SASS-NUOVA}
were false, we would deduce from~\eqref{ASS1-NUOVA} 
that~$Q_-\cap\EEEM=\varnothing$.
In particular, from \eqref{ZERO}, 
we obtain that
for almost every $x\in Q_-$
we have that~$u_\eps(x)\not\to-1$. Thus, \eqref{mono} implies $u_\eps(y)\not\to-1$, for almost every $y\in {\rm C}(d,\overline K)$. Hence $|\EEEM\cap {\rm C}(d,\overline K)|=0$,
which is in contradiction with~\eqref{ECE00-BIS}.

Having established~\eqref{SASS-NUOVA},
we use it to see that~$P:=(0,\dots,0,-\vartheta)\in \EEEM
\subseteq \LLM$, where~$\vartheta:=(K+\overline K)/2$,
and so
$$ -1=\lim_{\eps\to0^+}u_\eps(P)=\underline u(0),$$
which, by Maximum Principle implies that~$\underline u$
is constantly equal to~$-1$, which establishes~\eqref{8fd67sgh}.
This and Theorem~1.2 of~\cite{Tran}, we also obtain that~$u$ is~$1$D
if~$2\le n\le4$.

This proves the first statement in Theorem \ref{T04-BIS}.
The second follows from the first, applied to the function~$ -u(x',-x_n) $.
The third statement is then a combination of the first two ones:
more precisely,
claim~(i) in the last statement of
Theorem \ref{T04-BIS} follows from the previous two statements,
and then this implies claims~(ii) and~(iii), by
Theorem~1.3 of~\cite{Tran}.
\end{proof}

\begin{proof}[Proof of Theorem \ref{TTX}.]
Let us suppose that~$\{u=0\}$ is confined below
a complete graph (the case in which it is confined above being
analogous). 
We notice that
\begin{equation}\label{00}
\{u=0\}\ne\varnothing.
\end{equation}
To prove it, we argue by contradiction
and we suppose that~$u>0$ in the whole of~$\R^n$
(the case in which~$u<0$ is analogous). Then, by Theorem~2.1
in~\cite{FAR}, we have that~$u$ is identically equal to~$1$.
This is in contradiction with~\eqref{mono}
and so~\eqref{00} is proved.

{F}rom~\eqref{00} and~\eqref{mono} we obtain that
\begin{equation*}
\{ u>0\}\cap \{x_n>\gamma(x')\}\ne\varnothing.
\end{equation*}
Using this,
\eqref{02} 
and the fact that the set $\{x_n>\gamma(x')\}$ is open and path-connected,
we conclude that
\begin{equation*}
{\mbox{if $\quad x_n>\gamma(x')\quad$ then~$\quad u(x',x_n)>0$.}}
\end{equation*}
{F}rom this, we deduce that~$\overline u > 0$.
So we can apply once more
Theorem~2.1
in~\cite{FAR} (this time to the solution~$\overline u$ in~$\R^{n-1}$)
and deduce that~$\overline{u}$
is identical to~$1$. This establishes~\eqref{Tx1}.

Now assume in addition that~$2\le n\le 4$. Then~\eqref{Tx1}
and Theorem~1.2 of~\cite{Tran} give that~$u$ is~$1$D.
\end{proof}

\subsection{Geometric analysis of level sets}\label{S1}

Here we collect some auxiliary geometric results of somehow
elementary nature that will be useful for the proof of
Theorems~\ref{T1.0vecchio} and~\ref{T1.0vecchio-nuova}.
For this, we use the following additional notation.
Given~$x'\in\R^{n-1}$, we denote by~${\rm r}(x')$ the vertical
straight line through~$x'$, i.e.
$$ {\rm r}(x'):=\{ (x',t),\ t\in\R\}.$$
Also, for any $p\in\R^{n-1}$ and~$r>0$,
we denote by~$B^{n-1}_r(p)$ the $(n-1)$-dimensional ball 
of radius~$r$ centered at~$p$ and~$B_r^{n-1}:=B_r^{n-1}(0)$.
Also, when no confusion arises, we implicitly identify~$\R^{n-1}\times\{0\}$
and~$\R^{n-1}$ (i.e. points of the form~$(x',0)\in\R^{n-1}\times\{0\}$
and points of the form~$x'\in\R^{n-1}$).

\begin{lemma}\label{n1}
Let~$d>0$, $\Psi\subset\R^n$ and~$Z$ be the projection of~$\Psi$ onto~$\R^{n-1}$, that is
\begin{equation}\label{76fdjkw7119-Z1}
Z:= \{ x'\in\overline{ B^{n-1}_d}
{\mbox{ s.t. }} {\rm r}(x')\cap \Psi\ne\varnothing  \}.
\end{equation}
Suppose that
\begin{equation}\label{76fdjkw7119}
{\mbox{$\Psi\cap \{ x\in\R^n {\mbox{ s.t. }}
|x'|\le d\}$ is compact.}}
\end{equation}
Then~$Z$ is closed.
\end{lemma}

\begin{proof} 
The geometric situation of
Lemma~\ref{n1} 
is described in Figure~2.
\bigskip

\begin{center} 
\includegraphics[height=2in]{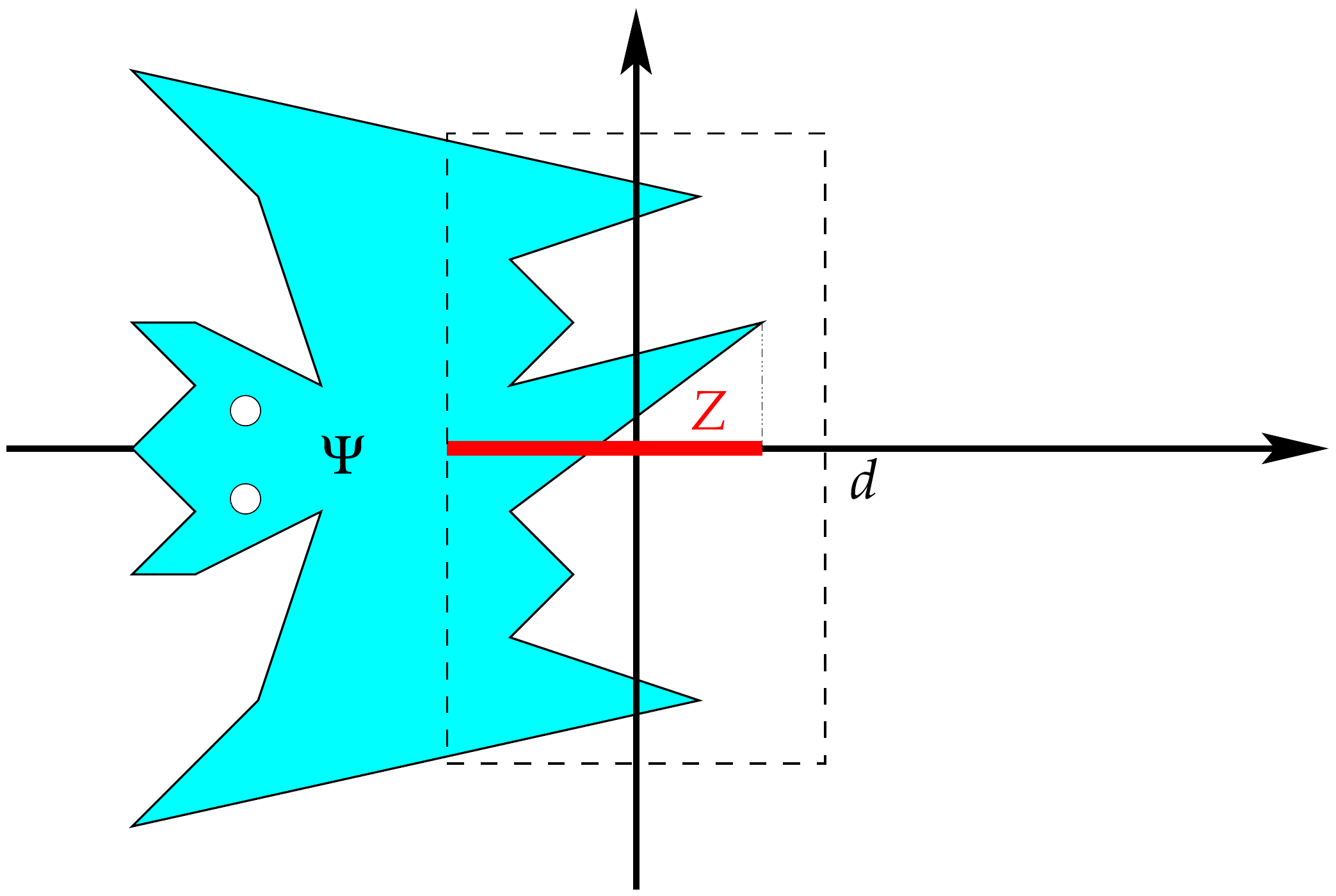}
\end{center}
\nopagebreak\centerline{\footnotesize\em
Figure~2. The sets of Lemma~\ref{n1}.}
\bigskip

The proof of Lemma~\ref{n1} goes as follows.
Let~$x'_k\in Z$ be a sequence approaching some~$x'_\star
\in\R^{n-1}$. Then,
there exists~$t_k$ such that~$(x'_k,t_k)\in {\rm r}(x'_k)\cap \Psi$.
By~\eqref{76fdjkw7119}, we can take a converging subsequence,
that is we can write
$$ \lim_{j\rightarrow+\infty} (x'_{k_j},t_{k_j})=
(x'_\star,t_\star)\in
\Psi\cap \{ x\in\R^n {\mbox{ s.t. }}
|x'|\le d\}
..$$
That is~$(x'_\star,t_\star)\in
{\rm r}(x'_\star)\cap \Psi$ and so~$
{\rm r}(x'_\star)\cap \Psi\ne\varnothing $. This
implies that~$x'_\star\in Z$, and so~$Z$ is closed.
\end{proof}

\begin{corollary}\label{3456789011122}
Let~$c\in\R$, $d>0$, $\overline K>K>0$, $\eta:=(\overline K-K)/4$, $v\in C^1(\R^n)$,
\begin{equation}\label{A8}\Gamma\subseteq{\rm C}(d,K)\end{equation}
and
\begin{equation}\label{7bis}
Z_\star := \Big\{ x'\in \overline{B^{n-1}_d}
{\mbox{ s.t. }} {\rm r}(x')\cap \{ v=c\}\ne\varnothing \Big\}.
\end{equation}
Suppose that
\begin{equation}\label{noem}
\{v=c\}\cap {\rm C}(d,\overline K)\ne\varnothing ,\end{equation}
that
\begin{equation}\label{d7ffg902093211}
\{ v=c\}\cap {\rm C}(d,\overline K)\subseteq \Gamma_\eta:=\bigcup_{p\in\Gamma} B_\eta(p)\end{equation}
and that
\begin{equation}\label{7dfs9709e993321a}
{\mbox{$\partial_n v(x)> 0$ for any $x\in\R^n$.}}\end{equation}
Then,~$Z_\star= \overline{B^{n-1}_d}$.
\end{corollary}

\begin{proof}
The geometric situation of
Corollary~\ref{3456789011122}
is described in Figure~3.
\bigskip

\begin{center} 
\includegraphics[height=2in]{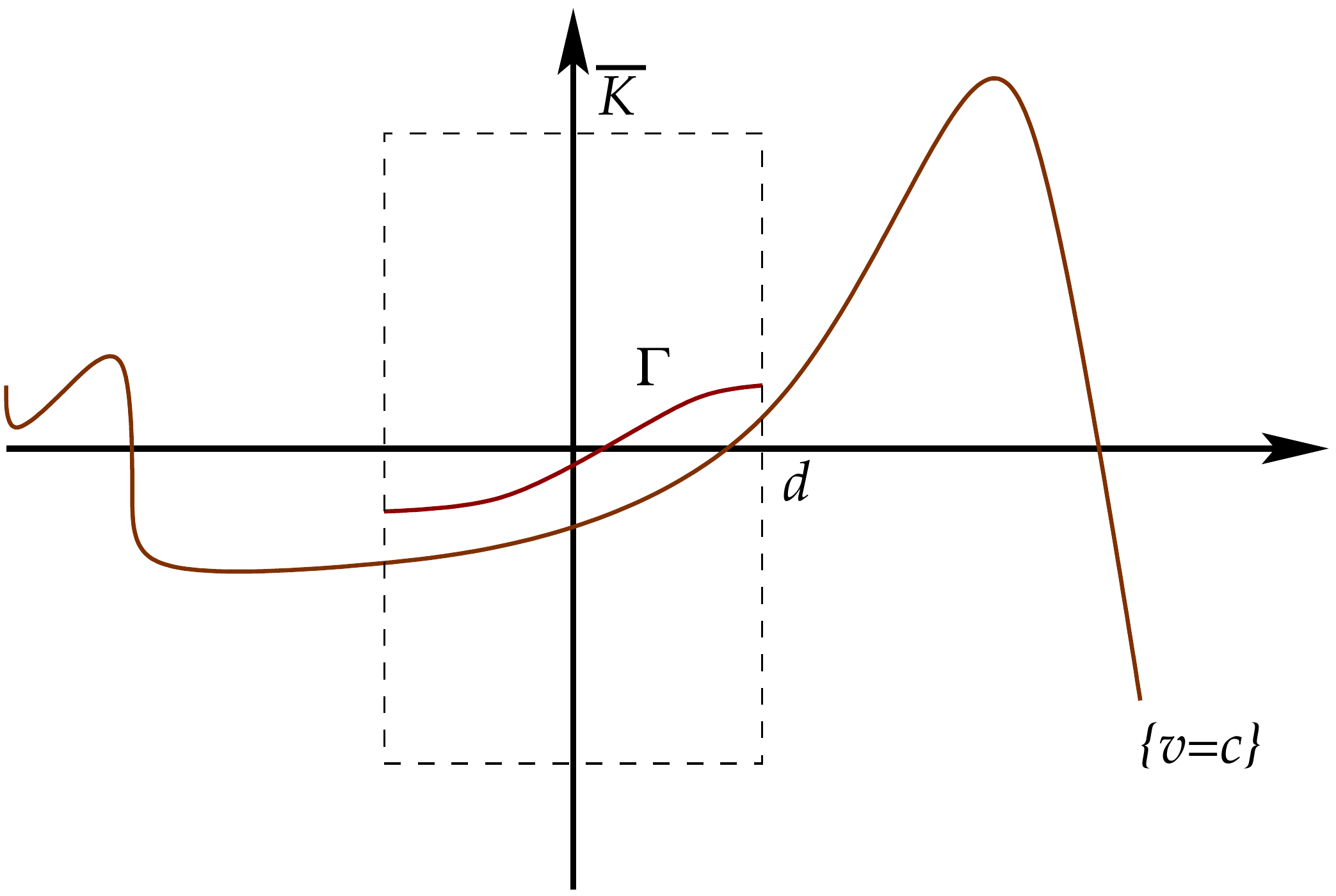}
\end{center}
\nopagebreak\centerline{\footnotesize\em
Figure~3. The sets in Corollary~\ref{3456789011122}.}
\bigskip

Its proof goes as follows.
We set~$\Psi:= \{ v=c\}\cap \{|x_n|\le \overline K\}$ and we observe that~$\Psi$
is closed since~$v$ is continuous.
We take $Z:=
\{ x'\in\overline{ B^{n-1}_d}
{\mbox{ s.t. }} {\rm r}(x')\cap \Psi\ne\varnothing  \}$.
Notice that this definition is coherent with~\eqref{76fdjkw7119-Z1}
and that
\begin{equation}\label{Asadf43123}
Z\ne\varnothing ,\end{equation}
thanks to~\eqref{noem}.

Accordingly, by Lemma~\ref{n1},
we have that
\begin{equation}\label{closed}
{\mbox{$Z$ is closed.}}\end{equation}
On the other hand, 
by~\eqref{A8}
and~\eqref{d7ffg902093211},
\begin{equation}\label{psi}\begin{split}
& \Psi\cap \{|x'| \le d\}
=\{v=c\}\cap {\rm C}(d,\overline K)
\\ &\qquad\subseteq \Gamma_\eta\cap {\rm C}(d,\overline K)
\subseteq 
{\rm C}(d,K+\eta).\end{split}\end{equation}
Now let~$p'\in Z$. Then, there exists $t\in\R$
such that $(p',t)\in \Psi$, hence~$v(p',t)=c$ and, by~\eqref{psi},
we have that~$|t|\le K+\eta$. Thus, by~\eqref{7dfs9709e993321a}
and the Implicit Function Theorem,
there exist~$\delta_1>0$ and~$\delta_2>0$
such that for any~$q'\in B_{\delta_1}^{n-1}(p')
\cap \overline{B^{n-1}_d}$
there exists~$t(q') \in (t-\delta_2,t+\delta_2)$ for which~$v(q',t(q'))=c$.
By possibly taking~$\delta_1$ smaller, we may and
do suppose that~$\delta_2<\eta$, therefore~$|t(q')|\le |t|+\delta_2
<K+2\eta\le\overline K$, which gives that~$q'\in Z$.
This says that
\begin{equation}\label{open}
{\mbox{$Z$ is also open in~$\overline{B^{n-1}_d}$.
}}\end{equation}
Accordingly, 
by~\eqref{closed}, \eqref{open} and~\eqref{Asadf43123}, we have 
that~$Z=\overline{B^{n-1}_d}$.
As a consequence,
using~\eqref{76fdjkw7119-Z1} and~\eqref{7bis}, we obtain
\begin{eqnarray*}
&& \overline{B^{n-1}_d} = Z =
\{ x'\in\overline{ B^{n-1}_d}
{\mbox{ s.t. }} {\rm r}(x')\cap \Psi\ne\varnothing \}
\\ &&\qquad\subseteq
\{ x'\in\overline{ B^{n-1}_d}
{\mbox{ s.t. }} {\rm r}(x')\cap \{v=c\}\ne\varnothing \}
=Z_\star\subseteq \overline{B^{n-1}_d},
\end{eqnarray*}
hence
we have proved the desired result.
\end{proof}

With this, we can now prove Theorems~\ref{T1.0vecchio}
and~\ref{T1.0vecchio-nuova}:

\begin{proof}[Proof of Theorem \ref{T1.0vecchio}.]
We will show that, given~$c\in(-1,1)$ as in the statement of Theorem~\ref{T1.0vecchio},
\begin{equation}\label{topvecchio}\begin{split}
&{\mbox{the level set~$\{u=c\}$ is a complete graph, i.e.}}
\\&{\mbox{for any fixed~$x'_o\in\R^{n-1}$, we have that~$r(x'_o)\cap\{u=c\}
\ne\varnothing$.}}
\end{split}\end{equation}
For this, we take~$c_-$, $c_+\in\R$ such that
$$ -1<c_-<c<c_+<1.$$
{F}rom~\eqref{lanuova} we know that 
there exist~$p\in \LLM
\cap {\rm C}(d,\overline{K})$
and~$q\in\LLP\cap
{\rm C}(d,\overline{K})$ such that~$u_\eps(p)$
approaches~$1$ and~$u_\eps(q)$ approaches~$-1$
as~$\eps\rightarrow 0^+$. In particular, if~$\eps$
is suitably small,~$u_\eps(p)>c_+$ and~$u_\eps(q)<c_-$.
Since~${\rm C}(d,\overline{K})$ is convex and~$u_\eps$ continuous,
this gives that there exists~$x(\eps)\in {\rm C}(d,\overline{K})$
such that~$u_\eps(x(\eps))=c$, that is
\begin{equation*}
\{u_\eps=c\} \cap {\rm C}(d,\overline{K})\ne\varnothing.
\end{equation*}
We remark that this implies~\eqref{noem} with~$v:=u_\eps$,
while, with this setting, condition~~\eqref{7dfs9709e993321a} comes from~\eqref{mono}.

Furthermore, if we take~$\Gamma:=(\partial \EEE)\cap {\rm C}(d,\overline{K})$,
we have that~\eqref{TRAPvecchio} implies~\eqref{A8},
and that~\eqref{aCCO} implies~\eqref{d7ffg902093211}.
Consequently, we can apply
Corollary~\ref{3456789011122} 
with~$v:=u_\eps$ and~$\Gamma:=(\partial \EEE)\cap {\rm C}(d,\overline{K})$,
and we conclude that, for small~$\eps$,
$$\eps x'_o\in \overline{B^{n-1}_d}=
\Big\{ x'\in \overline{B^{n-1}_d}
{\mbox{ s.t. }} {\rm r}(x')\cap \{ u_\eps=c\}\ne\varnothing \Big\}. $$
That is, ${\rm r}(\eps x'_o)\cap \{ u_\eps=c\}\ne\varnothing $.
This, by the definition of~$u_\eps$, proves~\eqref{topvecchio} and claim (i) of Theorem \ref{T1.0vecchio}. 

Then, claims~(ii)-(iv) of Theorem \ref{T1.0vecchio}
follow from claim (i)
and Theorem~1.3 of~\cite{Tran}.
\end{proof}

\begin{proof}[Proof of Theorem \ref{T1.0vecchio-nuova}.]

Both~$\EEE\cap 
{\rm C}(d,\overline{K})$
and~$({\mathcal{C}}\EEE)\cap {\rm C}(d,\overline{K})$ have positive measure by \eqref {ECEvecchio}, hence
using~\eqref{COC} and the a.e. convergence of~$u_\eps$,
we conclude that both~$\LLM\cap {\rm C}(d,\overline{K})$
and~$\LLP\cap {\rm C}(d,\overline{K})$ have positive measure. In particular, they are non-empty and~\eqref{lanuova}
is satisfied. With this, Theorem \ref{T1.0vecchio-nuova} is now a direct consequence
of Theorem~\ref{T1.0vecchio}.
\end{proof}

\section{Rigidity and symmetry from the limit varifold}\label{AG}

In this section we investigate the structure of the limit varifold, with the aim of
proving Theorems~\ref{T1} and~\ref{T2}.

\subsection{The limit varifold}\label{GAG}

Here we relate the structure of the limit varifold with the asymptotic
properties of the solutions. 

\begin{lemma}\label{Lemma 1}
Let~$u$ be a solution of~\eqref{PDE}
satisfying~\eqref{mono} and~\eqref{e bound}
and let~$V$ be the associated limit varifold.
Assume that there exists~$\overline x=(0,\dots,0,\overline x_n)$
with~$\overline x_n \in(0,h_o)$ that does not belong to~$V$.
Then
\begin{equation}\label{9djhrrfchg00} 
\lim_{x_n\rightarrow+\infty} u(x',x_n)=1.\end{equation}
Similarly, 
if there exists~$\underline x=(0,\dots,0,\underline x_n)$
with~$\underline x_n \in(-h_o,0)$ that does not belong to~$V$, then
$$ \lim_{x_n\rightarrow-\infty} u(x',x_n)=-1.$$
\end{lemma}

\begin{proof} We prove the first claim since the second one is alike.
Since~$V$ is closed
in~${\rm C}(d_o,h_o)$, the distance from~$\overline x$ to~$V$ is
strictly positive, therefore there exists~$\delta>0$ such that~$B_\delta(\overline x)\subseteq
{\rm C}(d_o,h_o)$ and
\begin{equation}\label{icvddffg12} B_\delta(\overline x)\cap V_\delta=\varnothing,\end{equation}
where the notation in~\eqref{6bis} has been used.
Using the uniform convergence of $u_\eps$ on each connected compact subset of 
$C(d_o, h_o) \setminus  V$ 
we immediately infer that either
$$ 0=\lim_{\varepsilon\to0^+} |u_\varepsilon(\overline x)-1|=
\lim_{\varepsilon\to0^+} \left| u\left( 0,\dots,0,\frac{\overline x_n}{\varepsilon}
\right)-1\right|=|\overline u(0)-1|$$
or
$$ 0=\lim_{\varepsilon\to0^+} |u_\varepsilon(\overline x)+1|=
\lim_{\varepsilon\to0^+} \left| u\left( 0,\dots,0,\frac{\overline x_n}{\varepsilon}
\right)+1\right|=|\overline u(0)+1|,$$
where the notation in~\eqref{profiles} and the assumption that~$\overline x_n>0$
have been used. Hence either~$\overline u(0)=1$ or~$\overline u(0)=-1$.
This and the Maximum Principle implies that~$\overline u$
is constantly equal to either~$1$ or~$-1$.

On the other hand, $\overline u$ cannot be constantly equal to~$-1$
(otherwise, by~\eqref{mono}, also~$u$ would be
constantly equal to~$-1$, thus contradicting~\eqref{mono} itself).
This says that~$\overline u$
is constantly equal to~$1$, which is~\eqref{9djhrrfchg00}.
\end{proof}

\subsection{Proof of the symmetry results from the behavior of the
limit varifold}\label{TGT}

Now we are ready to complete the proof of Theorems~\ref{T1} and~\ref{T2}.

\begin{proof}[Proof of Theorem \ref{T1}.]
By Lemma~\ref{Lemma 1}, we have that
$$ \lim_{x_n\rightarrow+\infty} u(x',x_n)=1 \
{\mbox{ and }}
\lim_{x_n\rightarrow-\infty} u(x',x_n)=-1,$$
which is~(i).
Then, claims~(ii) and~(iii) follow from~\cite{S}.
\end{proof}

\begin{proof}[Proof of Theorem \ref{T2}.]
Let~$\tilde x=(0,\dots,\tilde x_n)\in {r}_\star\setminus V$. We
observe that $ \tilde x_n \neq 0$,
since, by~\eqref{EMU} and~\eqref{LU0}, we know that~$ 0 \in V$.  


Combining this with Lemma~\ref{Lemma 1}, we see that 
$$ {\mbox{either }} \; 
\lim_{x_n\to+\infty} u(x',x_n)=1 \; 
{\mbox{ or }} \; \lim_{x_n\to-\infty} u(x',x_n)=-1 .$$ 
This and Theorem~1.2 of~\cite{Tran} give that~$u$ is~$1$D.  
\end{proof}


\begin{thebibliography}{dPKW11}

\bibitem[AAC01]{AAC}
Giovanni Alberti, Luigi Ambrosio, and Xavier Cabr{\'e}, \emph{On a
  long-standing conjecture of {E}. {D}e {G}iorgi: symmetry in 3{D} for general
  nonlinearities and a local minimality property}, Acta Appl. Math. \textbf{65}
  (2001), no.~1-3, 9--33, Special issue dedicated to Antonio Avantaggiati on
  the occasion of his 70th birthday. \MR{MR1843784 (2002f:35080)}

\bibitem[AC00]{AC}
Luigi Ambrosio and Xavier Cabr{\'e}, \emph{Entire solutions of semilinear
  elliptic equations in {${\mathbb R}\sp 3$} and a conjecture of {D}e
  {G}iorgi}, J. Amer. Math. Soc. \textbf{13} (2000), no.~4, 725--739
  (electronic). \MR{MR1775735 (2001g:35064)}

\bibitem[Ban89]{Bangert}
V.~Bangert, \emph{On minimal laminations of the torus}, Ann. Inst. H.
  Poincar\'e Anal. Non Lin\'eaire \textbf{6} (1989), no.~2, 95--138.
  \MR{90e:58021}

\bibitem[BCN97]{BCN}
Henri Berestycki, Luis Caffarelli, and Louis Nirenberg, \emph{Further
  qualitative properties for elliptic equations in unbounded domains}, Ann.
  Scuola Norm. Sup. Pisa Cl. Sci. (4) \textbf{25} (1997), no.~1-2, 69--94
  (1998), Dedicated to Ennio De Giorgi. \MR{MR1655510 (2000e:35053)}

\bibitem[DG79]{DG}
Ennio De~Giorgi, \emph{Convergence problems for functionals and operators},
  Proceedings of the International Meeting on Recent Methods in Nonlinear
  Analysis (Rome, 1978) (Bologna), Pitagora, 1979, pp.~131--188. \MR{MR533166
  (80k:49010)}

\bibitem[dPKW11]{delpinokowei}
Manuel del Pino, Micha{\l} Kowalczyk, and Juncheng Wei, \emph{On {D}e
  {G}iorgi's conjecture in dimension {$N\geq 9$}}, Ann. of Math. (2)
  \textbf{174} (2011), no.~3, 1485--1569. \MR{2846486 (2012i:35133)}

\bibitem[Far99]{Far99}
Alberto Farina, \emph{Symmetry for solutions of semilinear elliptic equations
  in {${\mathbb{R}}\sp N$} and related conjectures}, Ricerche Mat. \textbf{48}
  (1999), no.~suppl., 129--154, Papers in memory of Ennio De Giorgi.
  \MR{MR1765681 (2001h:35056)}

\bibitem[Far03]{FAR}
\bysame, \emph{Rigidity and one-dimensional symmetry for semilinear elliptic
  equations in the whole of {$\mathbb R\sp N$} and in half spaces}, Adv. Math.
  Sci. Appl. \textbf{13} (2003), no.~1, 65--82. \MR{MR2002396 (2004m:35078)}

\bibitem[FV08]{FV}
Alberto Farina and Enrico Valdinoci, \emph{{G}eometry of quasiminimal phase
  transitions}, Calc. Var. Partial Differential Equations \textbf{33} (2008),
  no.~1, 1--35.

\bibitem[FV09]{LE}
\bysame, \emph{The state of the art for a conjecture of {D}e {G}iorgi and
  related problems}, Recent progress on reaction-diffusion systems and
  viscosity solutions, World Sci. Publ., Hackensack, NJ, 2009, pp.~74--96.
  \MR{2528756}

\bibitem[FV11]{Tran}
\bysame, \emph{1{D} symmetry for solutions of semilinear and quasilinear
  elliptic equations}, Trans. Amer. Math. Soc. \textbf{363} (2011), no.~2,
  579--609. \MR{2728579 (2011j:35081)}

\bibitem[GG98]{GG}
N.~Ghoussoub and C.~Gui, \emph{On a conjecture of {D}e {G}iorgi and some
  related problems}, Math. Ann. \textbf{311} (1998), no.~3, 481--491.
  \MR{MR1637919 (99j:35049)}

\bibitem[GM88]{MR950604}
Morton~E. Gurtin and Hiroshi Matano, \emph{On the structure of equilibrium
  phase transitions within the gradient theory of fluids}, Quart. Appl. Math.
  \textbf{46} (1988), no.~2, 301--317. \MR{950604 (89j:49015)}

\bibitem[HT00]{HT}
John~E. Hutchinson and Yoshihiro Tonegawa, \emph{Convergence of phase
  interfaces in the van der {W}aals-{C}ahn-{H}illiard theory}, Calc. Var.
  Partial Differential Equations \textbf{10} (2000), no.~1, 49--84. \MR{1803974
  (2001m:49070)}

\bibitem[JGV09]{ZAMP}
Hannes Junginger-Gestrich and Enrico Valdinoci, \emph{Some connections between
  results and problems of {D}e {G}iorgi, {M}oser and {B}angert}, Z. Angew.
  Math. Phys. \textbf{60} (2009), no.~3, 393--401. \MR{MR2505410}

\bibitem[Mod85]{Mod}
Luciano Modica, \emph{A gradient bound and a {L}iouville theorem for nonlinear
  {P}oisson equations}, Comm. Pure Appl. Math. \textbf{38} (1985), no.~5,
  679--684. \MR{MR803255 (87m:35088)}

\bibitem[Mod87]{Modicark}
\bysame, \emph{The gradient theory of phase transitions and the minimal
  interface criterion}, Arch. Rational Mech. Anal. \textbf{98} (1987), no.~2,
  123--142. \MR{MR866718 (88f:76038)}

\bibitem[Sav09]{S}
Ovidiu Savin, \emph{Regularity of flat level sets in phase transitions}, Ann.
  of Math. (2) \textbf{169} (2009), no.~1, 41--78. \MR{MR2480601 (2009m:58025)}

\end{thebibliography}

\providecommand{\bysame}{\leavevmode\hbox to3em{\hrulefill}\thinspace}
\providecommand{\MR}{\relax\ifhmode\unskip\space\fi MR }
\providecommand{\MRhref}[2]{%
  \href{http://www.ams.org/mathscinet-getitem?mr=#1}{#2}
}
\providecommand{\href}[2]{#2}

\end{document}